\documentclass{amsart}
\usepackage{amsmath}
\usepackage{amsthm}
\usepackage{amssymb}
\usepackage{tikz-cd}
\usepackage[T1]{fontenc}
\usepackage[utf8]{inputenc}
\usepackage[english]{babel}
\usepackage[colorlinks=true,linkcolor=red]{hyperref}

\usepackage{tikz}
\usetikzlibrary{shapes}
\usepackage{float}
\usepackage[autostyle]{csquotes}
\usepackage[shortlabels]{enumitem}
\usepackage{xspace}
\usepackage{extarrows}
\usepackage{comment}
\usepackage{centernot}
\usepackage{mathtools}
\usepackage{cleveref}
\usepackage{amsmath}
\usepackage{svg}
\usepackage[titletoc]{appendix}

\usepackage[backend=biber, style=alphabetic, maxbibnames=99]{biblatex}

\theoremstyle{definition}
\newtheorem{defn}{Definition}[section]
\newtheorem{example}[defn]{Example}
\newtheorem{remark}[defn]{Remark}

\theoremstyle{plain}

\newtheorem{prop}[defn]{Proposition}
\newtheorem{thm}[defn]{Theorem}
\newtheorem{cor}[defn]{Corollary}

\newcommand\restr[2]{{
  \left.\kern-\nulldelimiterspace 
  #1 
  \vphantom{\big|} 
  \right|_{#2} 
  }}

\newcommand{\F}{\mathcal{F}}

\DeclareMathOperator{\Aut}{\mathrm{Aut}}
\DeclareMathOperator{\Homeo}{\mathrm{Homeo}}

\renewcommand{\Bbb}{\mathbb}

\newcommand{\Fraisse}{Fra\"{i}ss\'{e}\xspace}

\renewcommand{\L}{\mathcal{L}}

\usepackage{fancyhdr}

\usepackage[
  top=2cm,
  bottom=2cm,
  left=4cm,
  right=4cm,
  heightrounded,
]{geometry}

\theoremstyle{definition}

\theoremstyle{remark}
\newtheorem*{theorem*}{Theorem}
\newtheorem*{cor*}{Corollary}

\theoremstyle{definition}
\theoremstyle{definition}

\theoremstyle{definition}

\theoremstyle{remark}

\makeatletter
\providecommand*{\cupdot}{%
  \mathbin{%
    \mathpalette\@cupdot{}%
  }%
}
\newcommand*{\@cupdot}[2]{%
  \ooalign{%
    $\m@th#1\cup$\cr
    \sbox0{$#1\cup$}%
    \dimen@=\ht0 %
    \sbox0{$\m@th#1\cdot$}%
    \advance\dimen@ by -\ht0 %
    \dimen@=.5\dimen@
    \hidewidth\raise\dimen@\box0\hidewidth
  }%
}

\providecommand*{\bigcupdot}{%
  \mathop{%
    \vphantom{\bigcup}%
    \mathpalette\@bigcupdot{}%
  }%
}
\newcommand*{\@bigcupdot}[2]{%
  \ooalign{%
    $\m@th#1\bigcup$\cr
    \sbox0{$#1\bigcup$}%
    \dimen@=\ht0 %
    \advance\dimen@ by -\dp0 %
    \sbox0{\scalebox{2}{$\m@th#1\cdot$}}%
    \advance\dimen@ by -\ht0 %
    \dimen@=.5\dimen@
    \hidewidth\raise\dimen@\box0\hidewidth
  }%
}
\makeatother

\def\Ind#1#2{#1\setbox0=\hbox{$#1x$}\kern\wd0\hbox to 0pt{\hss$#1\mid$\hss}
\lower.9\ht0\hbox to 0pt{\hss$#1\smile$\hss}\kern\wd0}

\def\notind#1#2{#1\setbox0=\hbox{$#1x$}\kern\wd0
\hbox to 0pt{\mathchardef\nn=12854\hss$#1\nn$\kern1.4\wd0\hss}
\hbox to 0pt{\hss$#1\mid$\hss}\lower.9\ht0 \hbox to 0pt{\hss$#1\smile$\hss}\kern\wd0}

\usepackage[retainorgcmds]{IEEEtrantools}

\newcommand{\RUC}{\mathrm{RUC}}
\newcommand{\Age}{\mathrm{Age}}
\newcommand{\Stab}{\mathrm{Stab}}
\newcommand{\A}{\mathcal{A}}
\newcommand{\X}{\mathcal{X}}

\renewcommand{\phi}{\varphi}

\renewcommand{\phi}{\varphi}

\author{Alessandro Codenotti}
\title{Fixed points on null and tame flows for groups of automorphisms}

\begin{document}

\begin{abstract}
    Using a generalization of the Kechris-Pestov-Todor\v{c}evi\'{c} correspondence due to Nguyen Van Th\'{e} we obtain fixed point theorems for null and tame actions of groups of the form $\Aut(\F)$, where $\F$ is a \Fraisse structure. In particular we show that if $\Age(\F)$ is a free joint embedding class, then every null flow $\Aut(\F)\curvearrowright X$ has a fixed point, while if $\Age(\F)$ is a free amalgamation class, then every tame flow $\Aut(\F)\curvearrowright X$ has a fixed point.
\end{abstract}

\maketitle

\section{Introduction}

Given a topological group $G$, which we will always assume to be Hausdorff, a \emph{$G$-flow}, or simply a \emph{flow} when $G$ is clear from context, is a compact space $X$ with a continuous action $G\curvearrowright X$. A $G$-flow is called \emph{minimal} if $\overline{G\cdot x}=X$ for every $x\in X$, equivalently if there are no proper, closed, $G$-invariant subspaces. A $G$-flow $Y$ is a \emph{factor} of a $G$-flow $X$ if there there is a continuous surjection $f\colon X\to Y$ which is $G$-equivariant, that is $$f(g\cdot x)=g\cdot f(x),\,\text{for every }x\in X,g\in G.$$
A classical result of Ellis shows that every topological group $G$ has a \emph{universal minimal flow}, that is a minimal $G$-flow which has all minimal $G$-flows as factors, moreover this flow is unique up to isomorphism \cite[Theorem 2]{ellis1960universal} so that we can talk about \emph{the} universal minimal flow of $G$, denoted by $M(G)$. A topological group $G$ is called \emph{extremely amenable} if every $G$-flow has a fixed point, or equivalently if $M(G)$ is a singleton. A celebrated result, now known as the KPT correspondence, established in \cite{KPT} by Kechris, Pestov and Todor\v{c}evi\'{c}, characterizes which groups of the form $\Aut(\F)$, where $\F$ is a \Fraisse structure, are extremely amenable, in terms of a Ramsey property for $\Age(\F)$. The KPT correspondence was generalized by Nguyen van Th\'{e} \cite[Theorem 1.2]{NVT}, where fixed points on a specified class of $\Aut(\F)$-flows are obtained from an approximate  version of the Ramsey property of $\Age(\F)$ for a specified family of colourings (see Theorem \ref{thm: NVT} for a precise statement). We focus on the classes of \emph{null} and \emph{tame} $\Aut(\F)$-flows, those being two notion of well-behavedness for flows (defined precisely in Definitions \ref{defn: null system} and \ref{defn: tame system}). The study of tame dynamical systems was initiated by K\"{o}hler \cite{Kohler} (under the name of \emph{regular} systems), who obtained a dicothomy for the enveloping semigroups of dynamical systems consisting of interval maps, based on the Bourgain-Fremlin-Talagrand dichotomy for pointwise compact sets of Baire class 1 functions \cite{bourgain1978pointwise}. The dynamical version of this dicothomy has more recently been extended to arbitrary metrizable systems in \cite{GMHNSsystems}, based on the notion of \emph{tameness}, that is the nonexistence of continuous function whose orbit contain a set equivalent to the standard basis of $\ell^1$ (indeed the roots of the Bourgain-Fremlin-Talagrand dicothomy trace back to Rosenthal's groundbreaking characterization of Banach spaces not embedding $\ell^1$ \cite{rosenthal1978some}). The theory of tame dynamical systems has been greatly developed by various authors over the past few decades, see in particular \cite{Glasner4} and the reference therein for a survey. This has produced many equivalent characterization of tameness, the most important for us being the one established in \cite{KerrLi} in terms of the nonexistence of infinite independent sets (see Definition \ref{defn: tame system}). The notion of nullness, initially defined in terms of sequence entropy, but later characterized in terms of the nonexistence of arbitrarily large (but finite) independence sets (see Definition \ref{defn: null system}), is a natural strengthening of the notion of tameness. We introduce two variants of the Ramsey property for the age of a \Fraisse class, closely related to the classes of flows mentioned above, which we call the null and tame Ramsey properties. We then show that free joint embedding classes satisfy the approximate null Ramsey property, while free amalgamation classes satisfy the approximate tame Ramsey property. Using the result by Nguyen van Th\'{e} mentioned above we translate this combinatorial properties of $\Age(\F)$ into dynamical statements about $\Aut(\F)$, in particular we obtain the following results (see Theorem \ref{thm: FJEP implies no nontrivial null flows} and Theorem \ref{thm: Henson, Hrushovski and FJEP implies no tame flows} for complete statements).

\begin{thm}\label{thm: main thm1}
    Let $\mathcal F$ be a \Fraisse structure such that $\Age(\F)$ has the free joint embedding property. Any null flow $\Aut(\F)\curvearrowright X$ has a fixed point. 
\end{thm}

Since tameness is a strictly stronger property than nullness, the analogous result for tame flows requires stronger hypotheses:

\begin{thm}\label{thm: main thm2}
    Let $\F$ be a \Fraisse structure such that $\Age(\F)$ has the free amalgamation property. Any tame flow $\Aut(\F)\curvearrowright X$ has a fixed point. 
\end{thm}

Let us mention that a special case of the above results was already established by Glasner and Megrelishvili, although through completely different methods. When $\F$ is a countable set, that is the \Fraisse limit of the free amalgamation class of finite sets, $\Aut(\F)=S_\infty$, the group of permutations of the natural numbers. It is shown in \cite[Theorem 8.4]{Glasner4} that every tame flow $S_\infty\curvearrowright X$ has a fixed point or equivalently that $M_t(S_\infty)$ (see Definition \ref{defn: M_n(G) and M_t(G)}) is a singleton. Although not stated explicitely there the techniques of \cite{Glasner4} also apply to $\Aut(M)$, where $M$ is the random graph, another standard example of a \Fraisse structure arising from a free amalgamation class. 

\section*{acknowledgements}
Most of the content of this paper first appeared in chapter 5 of the author's PhD thesis, but Theorem \ref{thm: main thm1} and Theorem \ref{thm: main thm2} were only established for actions on zero-dimensional spaces. This is because free joint embedding classes in one case, and free amalgamation classes in the other case, were only shown to satisfy an exact Ramsey property, instead of its approximate version (see Theorem \ref{thm: NVT} for more context on how the zero-dimensionality assumption is related to the difference between an exact and an approximate Ramsey property). We are indebted to Tom\'{a}s Ibarluc\'{i}a who, as a referee for the author's thesis, suggested that the arguments used in the thesis to verify the exact Ramsey properties could in fact be used to verify an approximate version of the same properties, allowing us to establish Theorem \ref{thm: main thm1} and Theorem \ref{thm: main thm2} as stated here, for actions on a compact space of any dimension.

\section{Definitions and background}\label{section: defns and background}

\subsection{\Fraisse theory}
Let $\L=\{R_i\}_{i\in I}$ be a relational language, where $R_i$ has arity $n_i$ for every $i$. A countable $\L$-structure $M$ is called \emph{ultrahomogeneous} if whenever $A,B\subseteq M$ are finite substructures and $f\colon A\to B$ is an isomorphism, it is possible to extend $f$ to an automorphism of $M$. A countable, ultrahomogeneous structure is called a \emph{\Fraisse structure}. Given a \Fraisse structure $M$ we let $\Age(M)$ denote the class of finite structures that embed into $M$. Using the ultrahomogeneity of $M$, it is easy to check that $\Age(M)$ satisfies the following properties:

\begin{enumerate}
	\item $\Age(M)$ contains countably many structures, up to isomorphism.
	\item If $A\in\Age(M)$ and $B\subseteq A$, then $B\in\Age(M)$, this is known as the \emph{hereditary property}.
	\item If $A,B\in\Age(M)$, then there exists $C\in\Age(M)$ into which both $A$ and $B$ can be embedded, this is known as the \emph{joint embedding property}
	\item If $A,B,C\in\Age(M)$ and $f_1\colon A\to B$, $f_2\colon A\to C$ are two embeddings, then there exist $D\in\Age(M)$ and embeddings $g_1\colon B\to D$ and $g_2\colon C\to D$ such that $g_1\circ f_1=g_2\circ f_2$, this is known as the \emph{amalgamation property}.
\end{enumerate}

A class of finite structures satisfying the four properties above is called a \emph{\Fraisse class} and, by a classical result due to \Fraisse, there is a one-to-one correspondence between \Fraisse classes and \Fraisse structures. Given a \Fraisse class $\mathcal F$, the (unique up to isomorphism) \Fraisse structure $M$ with $\Age(M)=\mathcal F$ is called the \emph{\Fraisse limit of $\mathcal F$}. The \Fraisse limit of $\mathcal F$ is also characterized as the unique structure $M$ with the following extension property: for every $A,B\in\Age(\mathcal F)$ and every pair of embeddings $f_1\colon A\to M$, $f_2\colon A\to B$, there exists an embedding $g\colon B\to M$ with $g\circ f_2=f_1$.

Various strengthenings of the properties above have been considered, the following two will be particularly relevant for our purposes.

\begin{defn}
	Let $\F$ be a \Fraisse structure in the language $\L$. Given $A,B\in\Age(\F)$ let $A\sqcup B$ be the structure whose underlying set is the disjoint union of $A$ and $B$, and such that there are no new relations between the elements of $A$ and those of $B$. Explicitly $R_i^{A\sqcup B}(c_1,\ldots,c_{n_i})$ holds in $A\sqcup B$ if and only if either $c_1,\ldots,c_{n_i}\in A$ and $R_i^A(c_1,\ldots,c_{n_i})$ holds in $A$, or $c_1,\ldots,c_{n_i}\in B$ and $R_i^B(c_1,\ldots,c_{n_i})$ holds in $B$. We say that $\Age(\F)$ has the \emph{free joint embedding property} if for every $A,B\in\Age(\F)$, $A\sqcup B$ also belongs to $\Age(\F)$.
\end{defn}

Clearly the free joint embedding property implies the joint embedding property. Given $(A_i)_{i\in\omega}\subseteq\Age(\F)$, the structure $A=\bigsqcup_i A_i$ can be defined analogously. While it won't be an element of $\Age(\F)$, being infinite, it can embedded in $\F$ as long as $\Age(\F)$ has the free joint embedding property, since in that case $\Age(A)\subseteq\Age(\F)$.

\begin{defn}
	Let $\F$ be a \Fraisse structure. We say that $\Age(\F)$ has the \emph{free amalgamation property} if for all $A,B,C\in\Age(\F)$ and all pairs of embedding $f_1\colon A\to B$, $f_2\colon A\to C$, the structure $B\sqcup C/\sim$ is also in $\Age(\F)$, where $b\sim c$ if and only if there exists $a\in A$ with $b=f_1(a)$ and $c=f_2(a)$. 
\end{defn}

Note that the free amalgamation property implies the amalgamation property and that, since we're working with relation languages, the free amalgamation property implies the free joint embedding property (by taking $A$ to be the empty structure structure). As is standard we will also say that $\F$ has one of the properties defined above to mean that $\Age(\F)$ has the property in question.

\subsection{Subalgebras and compactifications}\label{subsection: subalgebras and compactifications}

Let $G$ be a topological group. Note that $G$-flows are always assumed to be compact, so we will say that $X$ is a $G$\emph{-space} if it is a not necessarily compact space $X$ with a continuous $G$-action. For clarity we will still often explicitly point out when a $G$-space is not assumed to be compact. A \emph{$G$-ambit $(X,x)$} is a $G$-flow $X$ together with a distinguished point $x$ with dense orbit. Given a family of $G$-ambits $(X_\lambda,x_\lambda)_{\lambda\in\Lambda}$ their \emph{supremum} $\bigvee_{\lambda\in\Lambda} (X_\lambda,x_\lambda)$ is the subflow of $\prod_{\lambda\in\Lambda}X_\lambda$ induced on the orbit closure of $(x_\lambda)_{\lambda\in\Lambda}$, together with the distinguished point $(x_\lambda)_{\lambda\in\Lambda}$.

\noindent Let $X$ be a (not necessarily compact) $G$-space. An embedding $\nu\colon X\to Y$ with dense image into a compact space $Y$ is called a \emph{$G$-compactification of $X$} if the action $G\curvearrowright X$ extends to an action $G\curvearrowright Y$. We say that a function $f\in C(X)$ \emph{comes from a $G$-compactification} $\nu\colon X\to Y$ if there is $\tilde f\in C(Y)$ such that $f=\tilde f\circ\nu$ (note that, if it exists, $\tilde f$ is unique). Given a $G$-space $X$, a bounded continuous function $f\colon X\to\mathbb R$ is called \emph{right uniformly continuous} if $$\forall\varepsilon>0\,\exists U\in\mathcal N_G\,\forall g\in U\,\forall x\in X\left(|f(g^{-1}x)-f(x)|<\varepsilon\right),$$ where $\mathcal N_G$ denotes the set of neighbourhoods of the identity in $G$. The collection of bounded right uniformly continuous functions on $X$ equipped with the $\|\cdot\|_\infty$-norm forms a normed algebra denoted by $\RUC_b(X)$. There is a correspondence between $G$-compactifications of $X$ and uniformly closed $G$-invariant subalgebras of $\RUC_b(X)$ that contain the constant functions. The subalgebra $\mathcal A_\nu\subseteq\RUC_b(X)$ associated to a $G$-compactification $\nu\colon X\to Y$ is $$\A_\nu=\{f\in C(X)\mid f\text{ comes from $\nu$}\}.$$
	Conversely given such an algebra $\A\subseteq\RUC_b(X)$ we obtain a $G$-compactification $X^\A$ of $X$ by letting $$X^\A=\{f\colon\A\to\Bbb R\mid f\text{ is a linear functional}\},$$ equipped with the weak* topology. Identifying linear functionals and maximal ideals of $\RUC_b(X)$ we have a canonical embedding $\nu_\A\colon X\to X^\A$ given by $x\mapsto p_x$, where $p_x$ is the maximal ideal defined by $p_x=\{f\in X^\A\mid f(x)=0\}.$ In this description basic open sets of $X^\A$ are of the form $$U_{f,\delta}=\{p\in X^\A\mid |\tilde f(p)|<\delta\},$$ where $f$ is a linear functional $\mathcal A\to\mathbb R$, $\delta>0$ and $\tilde f(p)$ is the unique $r\in\Bbb R$ such that $f-r\in p$. Note that if $f\colon X\to\mathbb R$ is continuous, then $\tilde f(p_x)=f(x)$, so that $\tilde f$ is the extension of $f$ to the $G$-compactification $X^\A$. 
	
	\noindent If $f$ comes from a $G$-compactification of $X$ then clearly $f\in\RUC_b(X)$, and the converse is true. Indeed for any $f\in\RUC_b(X)$ there is a minimal $G$-compactification of $X$ that $f$ comes from, called the cyclic $G$-space of $f$ and denoted by $X_f$. This corresponds to $X^{\langle f\rangle}$, the $G$-compactification associated to the closed (in the topology induced by the $\|\cdot\|_\infty$-norm) subalgebra of $\RUC_b(X)$ generated by $f$. Moreover if $\A$ is any closed $G$-invariant subalgebra of $\RUC_b(X)$ that contains the constant functions and $f$, the flow $X^\A$ has $X_f$ as a factor. A reference for the facts just mentioned is \cite[IV, §5]{deVries}, of particular relevance are Theorem 5.13 and Theorem 5.18.
	
	There is a continuous left action of $G$ on $\RUC_b(G)$ given by $(g\cdot f)(h)=f(g^{-1}h)$, where $\RUC_b(G)$ is equipped with the norm topology. There is also a right action of $G$ given by $(g\bullet f)(h)=f(hg)$. When $\RUC_b(G)$ is equipped with the pointwise convergence topology, then this action is continuous on $\overline{G\bullet f}$ (closure in the pointwise convergence topology) for every $f\in\RUC_b(X)$, even though it can fail to be continuous on the whole of $\RUC_b(G)$. To summarize to every $f\in\RUC_b(G)$ we have associated two ambits: $(G^{\langle f\rangle},1_G)$ and $(\overline{G\bullet f},f)$ which are isomorphic, as proved by Nguyen Van Thé in \cite[Proposition 3.9]{NVT}.

 	\subsection{Null and tame functions}

	\noindent The following characterization of null and tame functions is due to David Kerr and Hanfeng Li \cite{KerrLi}. 
	
	\begin{defn}
		Let $A_1,\ldots,A_k\subseteq X$. An \emph{independence set} for $\{A_1,\ldots,A_k\}$ is a set $I\subseteq G$ such that for all finite $J\subseteq I$ and all functions $\sigma\colon J\to k$ we have $$\bigcap_{j\in J}j^{-1}A_{\sigma(j)}\neq\varnothing.$$
	\end{defn}
	
	\begin{defn}
		A tuple $x=(x_1,\ldots,x_k)\in X^k$ is an \emph{$IN$-tuple} (\emph{$IN$-pair} in the case $k=2$) if for every product neighbourhood $U_1\times U_2\times\cdots\times U_k$ of $x$ and every $n\in\mathbb N$, the tuple $\{U_1,\ldots,U_k\}$ has an independence set of size at least $n$.
	\end{defn}
	
	\begin{prop}[\protect{\cite[Proposition 5.4(i)]{KerrLi}}]
		If $\{A_1,\ldots,A_k\}$ is a tuple of closed sets with arbitrarily large finite independence sets, then $A_1\times\cdots\times A_k$ contains an $IN$-tuple.
	\end{prop}
	
	\noindent We take the following as our definition of a null function. It is equivalent to other, more common, characterizations by \cite[Proposition 5.9]{KerrLi}.
	\begin{defn}\label{defn: null system}
		Let $X$ be a compact $G$-space. A function $f\in\RUC_b(X)$ is \emph{not null} if and only if there is an $IN$-pair $(x,y)\in X\setminus\Delta_X$ such that $f(x)\neq f(y)$. A dynamical system $(X,G)$ is called \emph{null} if every $f\in\RUC_b(X)$ is null.
	\end{defn}
	
	\noindent Consider now an arbitrary $G$-space $X$, not necessarily compact. We say that $f\colon X\to\Bbb R$ is \emph{null} if and only if there is a null $G$-compactification $\nu\colon X\to\tilde X$ that $f$ comes from.
	
	\begin{remark} As already observed before Fact 3.6 in \cite{Ibarlucia} it follows from the fact that factors of null flows are null \cite[Corollary 5.5]{KerrLi} that $f$ is null if and only if $X_f$ is null. Indeed if $\nu\colon X\to Y$ is any null $G$-compactification that $f$ comes from, then it has $X_f$ as a factor, so that the latter must be null.
	\end{remark}
	\begin{prop}[\protect{\cite[Fact 3.6]{Ibarlucia}}]\label{prop: Ibarlucia nullness}
		A function $f\in\RUC(X)$ is null if and only if there are no $r,s\in\Bbb R$ with $r<s$ such that for every $n\in\Bbb N$ one can find $(g_i)_{i<n}\subseteq G$ and $(x_I)_{I\subseteq n}\subseteq X$ such that 
        \begin{align*}
            f(g_ix_I)&<r \quad\text{if $i\in I$} \\
            f(g_ix_I)&>s \quad\text{if $i\not\in I$}.
        \end{align*}
	\end{prop}

	\noindent We have analogous results for tameness. 
	
	\begin{defn}
		A tuple $x=(x_1,\ldots,x_k)\in X^k$ is an $IT$-tuple ($IT$-pair in the case $k=2$) if for every product neighbourhood $U_1\times U_2\times\cdots\times U_k$ of $x$, the tuple $\{U_1,\ldots,U_k\}$ has an infinite independence set.
	\end{defn}
	
	\begin{prop}[\protect{\cite[Proposition 6.4(i)]{KerrLi}}]\label{prop: IT tuple from closed sets}
		If $\{A_1,\ldots,A_k\}$ is a tuple of closed sets with an infinite independence set, then $A_1\times\cdots\times A_k$ contains an $IT$-tuple.
	\end{prop}
	
	\noindent We take the following as our definition of a tame function. This is not how tameness is usually defined in the literature, but it is equivalent to more common definitions by \cite[Proposition 6.6]{KerrLi}
	\begin{defn}\label{defn: tame system}
		Let $X$ be a compact $G$-space. A function $f\in\RUC(X)$ is \emph{not tame} if and only if there is an $IT$-pair $(x,y)\in X\setminus\Delta_X$ such that $f(x)\neq f(y)$. A dynamical system $(X,G)$ is called \emph{tame} if every $f\in\RUC_b(X)$ is tame.
	\end{defn}

    \begin{example}
        Examples of null (hence also tame) flows include all continuous actions by homeomorphisms $G\curvearrowright X$ where $X$ is a dendrite \cite{GMtrees}. More concretely for any homeomorphism $\varphi\colon [0,1]\to[0,1]$ the induced $\mathbb Z$-action is null (hence tame). Examples of flows that are not null are given by actions with positive entropy, the classical example being the action $\Homeo(2^\omega)\curvearrowright 2^\omega$ of the homeomorphism group of the Cantor set on the Cantor set itself. It is more difficult to construct flows which are tame but not null, an example of such a $\mathbb Z$-flow is given in \cite[Example 6.7]{KerrLi}.
    \end{example}
	
	\noindent Consider now an arbitrary $G$-space $X$, not necessarily compact. We say that $f\colon X\to\Bbb R$ is tame if and only if there is a tame $G$-compactification $\nu\colon X\to\tilde X$ that $f$ comes from.
	
	\begin{remark}\label{remark: tameness can be checked on cyclic space}
		Analogously to the case of null functions, it follows from \cite[Corollary 6.5]{KerrLi}, which shows that factors of tame flows are tame, that $f\colon X\to\mathbb R$ is tame if and only if $G\curvearrowright X^{\langle f\rangle}$ is tame. Indeed if $\nu\colon X\to Y$ is any tame $G$-compactification that $f$ comes from, then it has $X_f$ as a factor, so that the latter must be tame.
	\end{remark}

	The following is a version of Proposition \ref{prop: Ibarlucia nullness} for tame functions.
	\begin{prop}\label{prop: tameness}
		Let $X$ be a (not necessarily compact) $G$-space. A function $f\in\RUC_b(X)$ is tame if and only if there are no $r,s\in\Bbb R$ with $r<s$ for which there exist $(g_i)_{i<\omega}\subseteq G$ and $(x_I)_{I\subseteq_{fin} \omega}\subseteq X$ such that 
        \begin{equation}\label{eq: tameness}
        \begin{aligned}
            f(g_ix_I)&<r \text{ if $i\in I$}\\
            f(g_ix_I)&>s \text{ if $i\not\in I$ and $i\leq \max I$,}
        \end{aligned}
        \end{equation}
        where $I\subseteq_{fin}\omega$ means that $I$ is a finite subset of $\omega$.
	\end{prop}
	\begin{proof}
		The proof is very similar to that of \cite[Fact 3.6]{Ibarlucia}, but we spell out the details for completeness. First suppose $f\colon X\to\Bbb R$ does not satisfy \eqref{eq: tameness}, we want to show that $f$ is not tame. Let $r<s\in\mathbb R$, $(g_i)_{i<\omega}\subseteq G$ and $(x_I)_{I\subseteq_{fin} \omega}\subseteq X$ witness the failure of \eqref{eq: tameness} for $f$. Extend $f$ to $\tilde f\colon X^{\langle f\rangle}\to\Bbb R$, we will show that $\tilde f$ is not tame, which immediately implies that $f$ is not tame either by Remark \ref{remark: tameness can be checked on cyclic space}. Consider \begin{align*}
		A_0&=\{p\in X^{\langle f\rangle}\mid \tilde f(p)\leq r\}\\
		A_1&=\{p\in X^{\langle f\rangle}\mid \tilde f(p)\geq s\}
		\end{align*}
		We claim that the set $Z=\{g_i\mid i<\omega\}$ is an infinite independence set for $\{A_0,A_1\}$. Indeed let $J\subseteq Z$ be finite and $\sigma\colon J\to 2$ an arbitrary function. Let $J_0=\{j\in\omega\mid g_j\in J\}$ and let $S_0=\{j\in J_0\mid \sigma(g_j)=1\}$. Let $I=J_0\setminus S_0\cup\{\max  J_0+1\}$. We will now verify that $$\bigcap_{s\in J}s^{-1}A_{\sigma(s)}\neq\varnothing,$$ by showing that $x_I\in s^{-1}A_{\sigma(s)}$ for every $s\in J$. Every such $s$ is of the form $g_i$ for some $i\in J_0$, we distinguish two cases based on whether $i\in S_0$ or $i\in J_0\setminus S_0$. If $i\in J_0\setminus S_0$, then $i\in I$, so that $f(g_ix_I)<r$ by \eqref{eq: tameness}. In particular $g_ix_I\in A_0$, so that $x_I\in g_i^{-1}A_0$. But $g_i^{-1}=s^{-1}$ and $\sigma(g_i)=0$ since $i\not\in S_0$, so that $x_I\in s^{-1}A_{\sigma(s)}$ as desired.

        If instead $i\in S_0$, then $i\not\in I$ and $i\leq\max I$, so that $f(g_ix_I)>s$ by \eqref{eq: tameness}. In particular $g_ix_I\in A_1$, so that $x_I\in g_i^{-1}A_1$. But $g_i^{-1}=s^{-1}$ and $\sigma(g_i)=1$ since $i\in S_0$, so that $x_I\in s^{-1}A_{\sigma(s)}$ in this case as well.
        
		By Proposition \ref{prop: IT tuple from closed sets} we can find an IT-pair $(x,y)\in A_0\times A_1$. Clearly $\tilde f(x)\neq\tilde f (y)$, so $\tilde f$ is not tame and neither is $f$.
		
		\noindent Conversely suppose that $f\colon X\to\Bbb R$ is not tame, we want to show that $f$ does not satisfy \eqref{eq: tameness}. Extend $f$ to a function $\tilde f\colon\tilde X\to\Bbb R$ on some $G$-compactification $\tilde X$ (for example we could take $\tilde X=X^{\langle f\rangle}$). Note that $\tilde f$ is necessarily not tame, which means that there is a nondiagonal IT-pair $(x,y)\in\tilde X^2$ such that $\tilde f(x)<\tilde f(y)$. Fix $r,s$ such that $\tilde f(x)<r<s<\tilde f(y)$ and, by continuity of $\tilde f$, open neighbourhoods $U_0\ni x$, $U_1\ni y$ such that $\tilde f(U_0)<r$, $\tilde f(U_1)>s$. Let $(g_i)_{i<\omega}$ be an infinite independence set for $\{U_0,U_1\}$, which exists since $(x,y)$ is an IT-pair. Now let $I$ be a finite subset of $\omega$ and let $\sigma\colon\max I\to 2$ be the complement of its indicator function, so that $\sigma(i)=0$ if and only if $i\in I$. Since $\{g_i\mid i<\omega\}$ is an independence set for $\{U_0,U_1\}$, we can find $$\tilde x_I\in\bigcap_{i\leq\max I}g_i^{-1}U_{\sigma(i)}.$$
		We claim that $r,s$, $(g_i)_{i<\omega}$ and $\tilde x_I$ as above witness that $\tilde f$ does not satisfy \eqref{eq: tameness}. 
        Indeed if $i\in I$, then $\sigma(i)=0$ and $\tilde x_I\in g_i^{-1}U_0$, so that $g_i\tilde x_I\in U_0$ and $f(g_i\tilde x_I)<r$. If instead $i\not\in I$ and $i\leq\max I$, then $\sigma(i)=1$ and $\tilde x_I\in g_i^{-1}U_1$, so that $g_i\tilde x_I\in U_1$ and $f(g_i\tilde x_I)>s$. 
    
        The only issue is that $\tilde x_I\in\tilde X$, while we were looking for points of $X$, but we can approximate $\tilde x_I$ with some $x_I\in X$, since we only need the values of $\tilde f(g_i\tilde x_I)$ and $f(g_ix_I)$ to be close for finitely many $i\leq\max I$.
		
	\end{proof}

	\begin{defn}\label{defn: M_n(G) and M_t(G)}
		Let $G$ be a topological group. We denote with $M_n(G)$ the \emph{universal minimal null flow of} $G$. This is a minimal null $G$-flow such that if $X$ is any other minimal null $G$-flow, there exists a $G$-equivariant surjection $M_n(G)\to X$. In other words $M_n(G)$ has all minimal null $G$-flows as factors. Existence and uniqueness of $M_n(G)$ follow from \cite[Theorem 1]{GutmanLi}: null flows are closed under products and subflows, so also under projective limits, by \cite[Corollary 5.5]{KerrLi}, while minimal flows are easily checked to be closed under projective limits. Similarly we let $M_t(G)$ denote the universal minimal tame flow of $G$, whose existence and uniqueness up to isomorphisms of $G$-flows can be established using Theorem 1 of \cite{GutmanLi} again, together with the fact that tame flows are also closed under products and subflows by \cite[Corollary 6.5]{KerrLi}.
	\end{defn}

    \section{The null and tame Ramsey properties}\label{section: null and tame Ramsey properties}
	\subsection{Colourings, Nullness, fixed points and free joint embedding}

    Before proceeding with this section, we need to fix some notational conventions, largely borrowed from \cite{NVT}. Given a \Fraisse structure $\F$ and a finite substructure $A\subseteq F$, we denote by $\Stab(A)$ the \emph{pointwise} stabilizer of $A$ in $G=\Aut(\F)$ and by $\binom \F A$ the set of \emph{embeddings} of $A$ in $\F$. With those conventions, thanks to the ultrahomogeneity of $\F$, we can identify $G/\Stab(A)$ and $\binom \F A$. 

    We will be interested in functions $\chi\colon\binom \F A\to\mathbb R$, which we will call \emph{colourings} or \emph{finite colourings} of $\binom \F A$, depending on whether their image is finite or not. If a colouring $\chi$ is constant on a set $S$, we will also say that $S$ is \emph{monochromatic for $\chi$}. Through the identification of $\binom \F A$ and $G/\Stab(A)$ outlined above we can think about a colouring $\chi\colon\binom \F A\to\mathbb R$ as a function $G/\Stab(G)\to\mathbb R$, which in turn can be identified with a function $\overline{\chi}\colon G\to\mathbb R$ which is constant on small enough balls of the right uniformity of $G$ (i.e. which is constant on cosets of $\Stab(A)$ for a big enough $A\subseteq \F$). We will confuse $\chi\colon\binom \F A\to\mathbb R$ and the corresponding $\overline{\chi}\colon G\to\mathbb R$ in the following, so that, for example, we will talk about colourings of $\binom \F A$ in $\RUC_b(G)$ and its subalgebras.
    
    Specializing the characterizations of nullness and tameness established in Proposition \ref{prop: Ibarlucia nullness} and Proposition \ref{prop: tameness} to the context of colourings we obtain the following corollary.

\begin{cor} \label{cor: not null/tame colouring}
	\noindent\begin{enumerate}
		\item 	A colouring $\chi\colon\binom{\F}{A}\to \Bbb R$ is \emph{not null} if and only if there are $r,s\in\Bbb R$ with $r<s$ such that for every $n<\omega$ there exist $(g_i)_{i<n}\subseteq \Aut(\F)$ and $(h_I)_{I\subseteq n}\subseteq \binom{\F}{A}$ such that
		\begin{align*}
			\chi(g_ih_I)&<r \quad\text{if $i\in I$} \\
			\chi(g_ih_I)&>s \quad\text{if $i\not\in I$}.
		\end{align*}
		
		\item A colouring $\chi\colon\binom{\F}{A}\to \Bbb R$ is \emph{not tame} if and only if there are $r,s\in\Bbb R$ with $r<s$ such that there exist sequences $(g_i)_{i<\omega}\subseteq\Aut(\F)$ and $(h_I)_{I\subseteq_{fin}\omega}\subseteq \binom{\F}{A}$ such that
		\begin{align*}
			\chi(g_ih_I)&<r \quad\text{if $i\in I$} \\
			\chi(g_ih_I)&>s \quad\text{if $i\not\in I$ and $i\leq\max I$}.
		\end{align*}
	\end{enumerate}
\end{cor}

    We need to borrow a few more definitions from \cite{NVT} before stating Theorem 1.2 of the same paper, which will be our main tool to translate combinatorial Ramsey theoretic properties into dynamical ones.

    \begin{defn}
        Let $\F$ be a \Fraisse structure and $\X$ a class of $\Aut(\F)$-flows. A pointed $\Aut(\F)$-flow $\Aut(\F)\curvearrowright (X,x)$ is called an \emph{ambit} if the orbit of $x$ is dense in $X$. We say that an ambit $\Aut(\F)\curvearrowright(X,x)$ is an \emph{$\X$-$\Aut(\F)$-ambit} if $\Aut(\F)\curvearrowright X\in\X$, that is if we obtain an element of $\X$ after forgetting the distinguished point.
    \end{defn}

    \begin{defn}
        Given $\A\subseteq\RUC_b(\Aut(\F))$ and $\varepsilon>0$ we say that $\F$ has the \emph{Ramsey property} (\emph{Ramsey property up to $\varepsilon$}) for finite colourings in $\A$ if for all $A,B\in\Age(\F)$, every finite set $C\subseteq \A$ of finite colourings of $\binom{\F}{A}$, there exists $h\in\binom{\F}{B}$ such that every $\chi\in C$ is constant (constant up to $\varepsilon$) on $h\circ \binom B A$. In particular when $\A$ is the subalgebra of null (tame) functions we say that $\F$ has the \emph{null (tame) Ramsey property (up to $\varepsilon)$}.
    \end{defn}

	\begin{thm}[\protect{\cite[Theorem 1.2]{NVT}}]\label{thm: NVT}
		Let $\F$ be a \Fraisse structure, $\X$ be a class of $\Aut(\F)$-flows such that the class of $\X$-$\Aut(\F)$-ambits is closed under suprema and factors, and that every $\Aut(\F)\curvearrowright X\in\X$ admits some $x\in X$ such that $\Aut(\F)\curvearrowright\overline{\Aut(\F)\cdot x}\in\X$. Then $$\A=\{f\in\RUC_b(G)\mid G\curvearrowright\overline{G\bullet f}\in\X\}$$ is a unital, left-invariant, closed subalgebra of $\RUC_b(\Aut(\F))$ and the following are equivalent:
		\begin{enumerate}
			\item Every $\Aut(\F)$-flow in $\X$ has a fixed point.
			\item For every $\varepsilon>0$, $\F$ has the Ramsey property up to $2\varepsilon$ for the finite colourings in $(\A)_\varepsilon$.
		\end{enumerate}
		Those imply the following equivalent statements:
		\begin{enumerate}\setcounter{enumi}{2}
			\item Every zero-dimensional $\Aut(\F)$-flow in $\X$ has a fixed point.
			\item $\F$ has the Ramsey property for the finite colourings in $\A$.
		\end{enumerate}
		When the finite colourings are dense in $\A$, all those statements are equivalent.
	\end{thm}

    Let us point out that in \cite{NVT}, the elements of $\RUC(X)$ and $\RUC_b(X)$ are complex valued functions, so that Theorem 1.2 of \cite{NVT} is stated in terms of $C^\ast$-subalgebras. In our setting the elements of $\RUC(X)$ and $\RUC_b(X)$ are real valued functions instead, so that Theorem \ref{thm: NVT} is stated in terms of closed subalgebras. This is not an important difference, since the real valued version can be proved exactly as the complex valued one, modulo replacing the correspondence between $G$-compactifications of $X$ and closed $C^\ast$-algebras of $\RUC_b(G,\mathbb C)$ with the correspondence between $G$-compactifications of $X$ and closed subalgebras of $\RUC_b(G,\mathbb R)$ outlined in Subsection \ref{subsection: subalgebras and compactifications}.

    We want to apply Theorem \ref{thm: NVT} when $\X$ is either the class of null $\Aut(F)$-flows or the class of tame $\Aut(\F)$-flows. Both of those classes are closed under suprema and factors by Corollary 5.5 and Corollary 6.5 of \cite{KerrLi}, respectively. The same corollaries also show that for any null (tame) $\Aut(\F)$-flow $X$ and any $x\in X$, $\Aut(\F)\curvearrowright\overline{\Aut(\F)\cdot x}$ is null (tame), so that all the hypothesis of Theorem \ref{thm: NVT} are satisfied, except at most the density hypothesis at the end. Moreover since $f\colon G\to\mathbb R$ is null (tame) if and only if $G\curvearrowright G^{\langle f\rangle}$ is null (tame), and the latter flow can be identified with $G\curvearrowright\overline{G\bullet f}$ as in Subsection \ref{subsection: subalgebras and compactifications}, when $\X$ is the class of null (tame) flows, the corresponding $\A$ is the algebra of null (tame) functions.

With the preceding paragraphs in mind, we obtain the following corollary by taking $\X$ to be the class of null, respectively tame, flows.

\begin{cor}\label{cor: NVT for null/tame}
    Let $\mathcal F$ be a \Fraisse structure, $G=\Aut(\mathcal F)$ and $\A\subseteq\RUC_b(G)$ the algebra of null (tame) functions. The following are equivalent:
	\begin{enumerate}
		\item Every null (tame) $G$-flow has a fixed point.
		\item For every $\varepsilon>0$, $\mathcal F$ has the Ramsey property up to $2\varepsilon$ for the finite colourings in $(\A)_\varepsilon$.
	\end{enumerate}
	Those imply the following equivalent statements:		\begin{enumerate}\setcounter{enumi}{2}
		\item Every zero-dimensional null (tame) $G$-flow has a fixed point.
		\item $\F$ has the Ramsey property for the finite colourings in $\A$. In other words $\mathcal F$ has the Ramsey property for the finite null (tame) colourings, since $\A$ is the algebra of null (tame) functions.
	\end{enumerate}
When the finite colourings are dense in $\A$, all those statements are equivalent.
\end{cor}

We can now verify that free joint embedding classes have the approximate null Ramsey property.

\begin{thm} \label{thm: FJEP implies null Ramsey property}
    Let $\mathcal F$ be a \Fraisse structure, $G=\Aut(\mathcal F)$, $\mathcal A\subseteq\RUC_b(G)$ the algebra of null functions and $\varepsilon>0$. If $\Age(\mathcal F)$ has the free joint embedding property, then $\mathcal F$ has the Ramsey property up to $2\varepsilon$ for the finite colourings in $(\mathcal A)_\varepsilon$.
\end{thm}
\begin{proof}
    We prove the contrapositive. Let $A,B\in\Age(\F)$ and $\chi_1,\ldots,\chi_n$ be colourings of $\binom \F A$ with $k$ colours such that there is no $h\in\binom \F B$ such that $h\circ\binom BA$ is monochromatic up to $2\varepsilon$ for all of the colourings $\chi_i$ at the same time. We want to show that at least one of the colourings $\chi_i$ is not in $(\mathcal A)_\varepsilon$. Consider the structure $D=\bigsqcup_{i<\omega} B_i$, with $B_i\cong B$ for all $i$, which can be embedded in $\F$ as mentioned above. We fix an embedding $g\colon D\to \F$ and, by restricting it to the various $B_i$'s, we obtain infinitely many embeddings of $B$ into $\F$ with pairwise disjoint images and with pairwise no relations between the images, formally let $$H''=\left\{h''\in\binom \F B\,\middle|\, \exists i\in\omega\quad h''=g\upharpoonright B_i\right\}.$$ 
    
    By assumption for every $h''\in H''$ there exists $1\leq j\leq n$ such that $\chi_j$ has oscillation $>2\varepsilon$ on $h''\circ\binom BA$. Since there are finitely many colourings while $H''$ is infinite, the pigeonhole principle gives $1\leq i\leq n$ such that for infinitely many $h''\in H''$, $\chi_i$ has oscillation $>2\varepsilon$ on $h''\circ \binom B A$. To simplify the notation we write $\chi$ for $\chi_i$ in the following. Our goal is to show that $\chi$ is not in $(\mathcal A)_\varepsilon$ by using the characterization in Proposition \ref{prop: Ibarlucia nullness}. Let $$H'=\left\{h''\in H''\,\middle|\, \chi \text{ is not constant up to $2\varepsilon$ on } h''\circ\binom B A\right\}.$$
    
    Since $\chi$ has finite image while $H'$ is infinite, another application of the pigeonhole principle yields $k_0<k_1\in\mathbb R$ such that for infinitely many $h'\in H'$, both $k_0$ and $k_1$ are in $\chi\left(h'\circ\binom B A\right)$ and $k_1-k_0>2\varepsilon$. Let
    $$H=\left\{h'\in H'\,\middle|\, k_0,k_1\in\chi\left(h'\circ\binom B A\right)\right\}.$$
    
    Now fixed $n\in\omega$, we can find in $H$ pairwise distinct embeddings $h_I\colon B\to \F$ for all $I\subseteq n$, since $H$ is infinite. Note that, by construction, $h_I(B)\cap h_{I'}(B)=\varnothing$ for all $I\neq I'\subseteq n$, and that the substructure of $\F$ with domain $h_I(B)\cup h_{I'}(B)$ is isomorphic to $B\sqcup B$.
    
    By assumption for every $I\subseteq n$ we can now find a pair of embeddings $h_I^0,h_I^1\colon A\to B$ such that $\chi(h_I\circ h_I^i)=k_i$ for $i\in\{0,1\}$.
    For $i<n$ we define a partial isomorphism $\widetilde{g}_i$ of $F$ with domain $\bigcup_{I\subseteq n}h_Ih^1_I(A)$ as follows. For every $I\subseteq n$, $\widetilde{g}_i\circ h_I\circ h^1_I=h_I\circ h^0_I$ if $i\in I$ and $\widetilde{g}_i\circ h_I\circ h^1_I=h_I\circ h^1_I$ if $i\not\in I$. That is, for $I\subseteq n$ and $x\in h_Ih^1_I(A)$, we have 
    $$\widetilde{g}_i(x)=\begin{cases}h_Ih_I^0(h_Ih_I^1)^{-1}(x) & \text{ if $i\in I$} \\
    	x & \text{ if $i\not\in I$.}\end{cases}$$
    
    Note that there is indeed such a partial isomorphism of $F$: for every $I\subseteq n$, $\widetilde{g_i}$ is either mapping $h_I\circ h^1_I(A)$ to $h_I\circ h^0_I(A)$, if $i\in I$, or is the identity on $h_I\circ h^1_I(A)$, if instead $i\not\in I$. Since the embeddings where chosen so that the substructure of $\F$ with underlying set $h_I(B)\cup h_{I'}(B)$ is isomorphic to $B\sqcup B$ for all $I\neq I'\subseteq n$, we can choose how $\widetilde{g}_i$ behaves on $I$ independently from the choices made for $I'\neq I$. 
	Finally let, for $i<n$, $g_i\colon \F\to \F$ be any automorphism of $\F$ extending the partial isomorphism $\widetilde{g_i}$ defined above. Suppose now that $f\in\RUC_b(G)$ is a function with $\|f-\chi\|_\infty<\varepsilon$. We want to show that $f$ cannot be null, which is equivalent to saying that $\chi\not\in(\mathcal A)_\varepsilon$. Given $n\in\omega$ let $x_I$ for $I\subseteq n$, and $g_i$ for $i<n$ be as just constructed, and take $r,s$ such that $k_0<r<s<k_1$ and $|k_0-r|>\varepsilon$, $|k_1-s|>\varepsilon$. We claim that those choices of $x_I$, $g_i$, $r$ and $s$ witness that $f$ is not null by the characterization in Proposition \ref{prop: Ibarlucia nullness}. Indeed if $i\in I$ we have 
	$$\chi(g_ix_I)=\chi(g_i\circ h_I\circ h_I^1)=\chi(h_I\circ h_I^0)=k_0,$$ which implies that $f(g_ix_I)<r$, since $|f(x_ig_I)-\chi(g_ix_I)|<\varepsilon$. If instead $i\not\in I$ we have $$\chi(g_ix_I)=\chi(g_i\circ h_I\circ h_I^1)=\chi(h_I \circ h_I^1)=k_1,$$ which implies that $f(g_ix_I)>s$, since $|f(g_ix_I)-\chi(g_ix_I)|<\varepsilon$, showing that $f$ is not null by the characterization in Proposition \ref{prop: Ibarlucia nullness}. In other words $\chi\not\in(\mathcal A)_\varepsilon$, concluding the proof.
\end{proof}

    The following is now an immediate consequence of the equivalence between $(1)$ and $(2)$ in Corollary \ref{cor: NVT for null/tame}.
 
	\begin{thm}\label{thm: FJEP implies no nontrivial null flows}
		Let $\mathcal F$ be a \Fraisse class such that $\Age(\F)$ has the free joint embedding property. Then every null $\Aut(\mathcal F)$-flow has a fixed point. Equivalently  $M_n(\Aut(\mathcal F))$ is trivial.
	\end{thm}
    
	\subsection{Tameness, fixed points, the Henson property and the Hrushovski property}
	By the previous section we already know that the \Fraisse classes of finite graphs and finite sets have the null Ramsey property, since they are free joint embedding classes. In this section we show that in fact they have the tame Ramsey property. 
	
	\begin{defn}
		Let $\F$ be a \Fraisse class. We say that $\F$ has the \emph{Henson property} if for every countable $A\subseteq \F$, there is $h\in\binom{\F}{A}$ such that if $g$ is any automorphism of $h(A)$, $g$ extends to an automorphism of $\F$. 
	\end{defn}

    Before proceeding with the main results of this section we go on an historical detour concerning the Henson property. 

    \begin{remark} \label{rem: FA implies Henson}
    We could not find whether the Henson property has been named before in the literature, so we adopted this name, since it was first shown to hold for the random graph and the $K_n$-free random graphs (also called the Henson graphs) by Henson \cite[Theorem 3.1]{Henson}. A natural strengthening, which we will call the \emph{strong Henson property} would be to require that every automorphism of $h(A)$ extends uniquely to an automorphism of $\F$ (indeed in \cite{Henson} the strong Henson property is established for the graphs mentioned above). The strong Henson property has been established for all free amalgamation classes in the PhD thesis \cite{Bilge} and later expanded to the wider class of \Fraisse classes with a stationary independence relation in \cite{Mueller}. There exists in the literature a similar but distinct notion of \emph{universal \Fraisse structure $\F$}, that is a \Fraisse structure such that for every $A\subseteq \F$, $\Aut(F)$ contains a copy of $\Aut(A)$. For examples all \Fraisse structures that admit Kat\v{e}tov functors are universal \cite{kubis2017katvetov}.
    \end{remark}

	\begin{defn}
		Let $F$ be a \Fraisse structure. We say that $\F$ has the \emph{Hrushovski property} or that $\Age(\F)$ is an \emph{EPPA class} (where EPPA stands for the \emph{extension property for partial automorphisms}) if for every $A\in\Age(\F)$ there exists a $B\in\Age(\F)$ and an embedding $h\colon A\to B$ such that for every partial isomorphism $p\colon A_1\to A_2$ of $h(A)$, $p$ extends to an automorphism of $B$. We call the embedding $h\colon A\to B$ an EPPA witness for $A$. 
	\end{defn}

We are now ready to prove the tame Ramsey property for the appropriate class of \Fraisse structures, from which we will then deduce dynamical consequences for their automorphism groups.

\begin{thm}\label{thm: Henson+Hrushovski+FJEP implies tame Ramsey property}   
	Let $\F$ be a \Fraisse structure, $G=\Aut(\F)$, $\mathcal A\subseteq\RUC_b(G)$ the algebra of tame functions and $\varepsilon>0$. If $\F$ has the Henson and the Hrushovski properties and $\Age(\F)$ has the free joint embedding property, then $\F$ has the Ramsey property up to $2\varepsilon$ for the finite colourings in $(\mathcal A)_\varepsilon$. 
\end{thm}
\begin{proof} 
    We prove the contrapositive. Let $A,B\in\Age(\F)$ and $\chi_1,\ldots,\chi_n$ be colourings of $\binom \F A$ with $k$ colours such that no $h\in\binom \F B$ is monochromatic up to $2\varepsilon$ for all of the colourings $\chi_i$ at the same time. We want to show that at least one of the colourings $\chi_i$ is not in $(\A)_\varepsilon$. Let $f\colon B\to C$ be an EPPA witness for $B$ and note that, since $\binom C B$ is nonempty, there is no $h\in\binom \F C$ which is monochromatic for all the colourings $\chi_i$ at the same time. Take countably many copies $B_i\cong B$, $C_i\cong C$ and embeddings $f_i\colon B_i\to C_i$, for $i<\omega$, such that, for every $i<\omega$, $f_i\colon B_i\to C_i$ is an EPPA witness for $B_i$. Let $D=\bigsqcup_{i<\omega} C_i$ and note that, since $\Age(\F)$ has the free joint embedding property, $D$ is a countable structure embeddable in $\F$. Since $\F$ has the Henson property we can find an embedding $g\colon D\to F$ such that for all $h\in\Aut(g(D))$, $h$ extends to an automorphism of $\F$. The embedding $g$ also gives infinitely many embeddings of $B$ in $\F$ with pairwise disjoint images and pairwise no relations between the images. Explicitly let $$H'=\left\{h'\in\binom \F B\,\middle|\, \exists i\in\omega\quad h'=g\upharpoonright f_i(B_i)\right\}.$$

    As in the proof of Theorem \ref{thm: FJEP implies null Ramsey property}, by applying the pigeonhole principle twice, we can find $i\leq n$ and $k_0<k_1$ with $k_1-k_0>2\varepsilon$ such that for $\chi=\chi_i$ the set $$H=\left\{h'\in H'\,\middle|\, k_0,k_1\in\chi\left(h'\circ\binom B A\right)\right\}$$ is infinite, so that we can find, for $I\subseteq_{\mathrm{fin}}\omega$, pairwise distinct embeddings $h_I\colon B\to \F$ in $H$. Note that, by construction, every $h_I$ can be extended to an embedding $\widetilde{h_I}\colon C\to \F$, such that the inclusion of $h_I(B)$ in $\widetilde{h_I}(C)$ is an EPPA witness for $h_I(B)$. Moreover we can find, for all $I\subseteq_{\mathrm{fin}}\omega$, embeddings $h_I^0,h_I^1\colon A\to B$ such that $\chi(h_I\circ h_I^i)=k_i$, $i\in\{0,1\}$.
    
    As in the proof of Theorem \ref{thm: FJEP implies null Ramsey property}, for every $I\subseteq_{\mathrm{fin}}\omega$, we can find a partial automorphism $g_I$ of $h_I(B)$, defined on $h_I(A)$ such that $g_I\circ h_I\circ h_I^1=h_I\circ h_I^0$. Since $\widetilde{h_I}(C)$ is an EPPA witness for $h_I(B)$, $g_I$ extends to an element of   $\Aut(\widetilde{h_I}(C))$, which we also call $g_I$. Consider now, for $i<\omega$, the automorphism $\widetilde{g_i}\colon g(D)\to g(D)$ given by $\widetilde{g_i}=\bigsqcup_I g_I'$, where $g_I'\colon \widetilde{h_I}(C)\to \widetilde{h_I}(C)$ is $g_I$ if $i\in I$ and is the identity of $\widetilde{h_I}(C)$ if $i\not\in I$. Note that, by construction, $\widetilde{g_i}\in\Aut(g(D))$, as there are no relations between the copies of $C$ in $D$. Finally, since $g$ was chosen to be an embedding as given by the Henson property, let $g_i$ extend $\widetilde{g_i}$ to $F$. Define $x_I=h_I\circ h_I^1$ and suppose $f\in\RUC_b(G)$ is such that $\|f-\chi\|_\infty<\varepsilon$. We want to show that $f$ is not tame by the characterization established in Proposition \ref{prop: tameness}. Take $r,s$ with $k_0<r<s<k_1$, $|k_0-r|>\varepsilon$ and $|k_1-s|>\varepsilon$. We claim that $r,s,x_I$ and $g_i$ as just constructed witness that $f$ is not tame. Indeed if $i\in I$ we have
    $$\chi(g_ix_I)=\chi(g_i\circ h_I\circ h_I^1)=\chi(g_I\circ h_I\circ h_I^1)=\chi(h_I\circ h_I^0)=k_0,$$
    which implies that $f(g_ix_I)<r$, since $|f(g_ix_I)-\chi(g_ix_I)|<\varepsilon$. If instead $i\not\in I$ we have
    $$\chi(g_ix_I)=\chi(g_i\circ h_I\circ h_I^1)=\chi(\mathrm{Id}_{\widetilde{h_I}(C)}\circ h_I\circ h_I^1)=\chi(h_I\circ h_I^1)=k_1,$$
    which implies that $f(x_ig_I)>s$, since $|f(g_ix_I)-\chi(g_ix_I)|<\varepsilon$, showing that $f$ is not tame by the characterization in Proposition \ref{prop: tameness}.
    
\end{proof}

The following is now an immediate consequence of the equivalence between $(1)$ and $(2)$ in Corollary \ref{cor: NVT for null/tame}.

 \begin{thm}\label{thm: Henson, Hrushovski and FJEP implies no tame flows}
     Let $\F$ be a \Fraisse structure such that $F$ has the Henson property and the Hrushovski property and such that $\Age(\F)$ has the free joint embedding property. Then every tame $\Aut(\F)$-flow has a fixed point. Equivalently $M_t(\Aut(\F))$ is a singleton.
 \end{thm}

As mentioned in Remark \ref{rem: FA implies Henson} all free amalgamation classes have the Henson property. By a result of Siniora and Solecki \cite[Corollary 4.6]{SinioraSolecki} free amalgamation classes also have the Hrushovski property. Together with the easy observation that, for the age of a \Fraisse structure in a finite relational language, the free amalgamation property implies the free joint embedding property we immediately obtain the following corollary from Theorem \ref{thm: Henson, Hrushovski and FJEP implies no tame flows}.
 
 \begin{cor}\label{thm: free amalgamation implies no tame flows}
    Let $\F$ be a \Fraisse structure such that $\Age(\F)$ has the free amalgamation property. Then every tame $\Aut(\F)$-flow has a fixed point. Equivalently $M_t(\Aut(\F))$ is trivial.
\end{cor}
	
\printbibliography

\end{document}